\documentclass[a4paper, 11pt]{article}
\usepackage[mathscr]{eucal}
\usepackage{amssymb}
\usepackage{latexsym}
\usepackage{amsthm}
\usepackage{amsmath}
\usepackage[dvips]{graphicx}
\usepackage{psfrag}
\usepackage{a4wide}

\newcommand{\R}{\mathbb{R}}

\newcommand{\Q}{\mathbb{Q}}
\newcommand{\Z}{\mathbb{Z}}

\newcommand{\p}{\mathfrak{p}} 
\newcommand{\Oc}{\mathcal{O}} 
\newcommand{\la}{\langle}
\newcommand{\ra}{\rangle}

\newtheorem{theorem}{Theorem}[section]

\newtheorem{lemma}[theorem]{Lemma}

\theoremstyle{definition}

\newtheorem{remark}[theorem]{Remark}

\makeatletter
\@addtoreset{equation}{section}

\makeatother

\title{Euclidean sets with only one distance modulo a prime ideal}
\author{
Hiroshi Nozaki\thanks{Department of Mathematics Education, 
	Aichi University of Education, 
	1 Hirosawa, Igaya-cho, 
	Kariya, Aichi 448-8542, 
	Japan. {\tt hnozaki@auecc.aichi-edu.ac.jp}}}

\begin{document}

\maketitle

\renewcommand{\thefootnote}{\fnsymbol{footnote}}

\begin{abstract}
Let $X$ be a finite set in the Euclidean space $\mathbb{R}^d$. If the squared distance between any two distinct points in $X$ is an odd integer, then 
the cardinality of $X$ is bounded above by $d+2$, as shown by Rosenfeld (1997) or Smith (1995). 
They proved that there exists a $(d+2)$-point set $X$ in $\mathbb{R}^d$ having only odd integral squared distances if and only if $d+2$ is congruent to $0$ modulo $4$. 
The distances can be interpreted as an element of the finite field  $\mathbb{Z}/2\mathbb{Z}$. We generalize this result for a local ring $(A_\mathfrak{p},\mathfrak{p}A_\mathfrak{p})$ as follows. 
Let $K$ be an algebraic number field that can be embedded into $\mathbb{R}$. Fix an embedding of $K$ into $\mathbb{R}$, and $K$ is interpreted as a subfield of $\mathbb{R}$. 
Let $A=\Oc_K$ be the ring of integers of $K$, and $\mathfrak{p}$ a prime ideal of $\mathcal{O}_K$. Let $(A_\mathfrak{p},\mathfrak{p}A_\mathfrak{p})$ be the local ring obtained from the localization $(A\setminus \mathfrak{p})^{-1} A$, which is interpreted as a subring of $\mathbb{R}$. 
If the squared distances of $X\subset \mathbb{R}^d$ are in $A_\p$ and each squared distance is congruent to some constant $k \not\equiv 0 $ modulo $\p A_\p$, then $|X| \leq d+2$, as shown by Nozaki (2023). 
In this paper, 
we prove that there exists a set $X\subset \mathbb{R}^d$ attaining the upper bound $|X| \leq d+2$ if and only if $d+2$ is congruent to $0$ modulo $4$ when the finite field $A_\mathfrak{p}/ \mathfrak{p} A_\mathfrak{p}$ is of characteristic 2, and $d+2$ is congruent to $0$ modulo $p$ when $A_\mathfrak{p}/ \mathfrak{p} A_\mathfrak{p}$ is of characteristic $p$ odd. We also provide examples attaining this upper bound.

\end{abstract}


\section{Introduction}
For a finite set $X \subset \mathbb{R}^d$, define
\[
D(X)=\{||x-y||^2 \colon\, x,y \in X, x \ne y\},
\]
where $||x-y||$ is the Euclidean distance between $x$ and $y$. 
A set $X \subset \mathbb{R}^d$ is called {\it an $s$-distance set} if $|D(X)|=s$. For an $s$-distance set $X \subset \R^d$, there is an upper bound $|X| \leq \binom{d+s}{s}$ \cite{BBS83}. 
An $s$-distance set attaining this upper bound is said to be {\it tight}. 
The major problem for $s$-distance set is to determine the largest possible $s$-distance sets for given $d$ and $s$. 
Largest $s$-distance sets in $\mathbb{R}^d$ are known only for $(s, d) = (1, \text{any}),(\leq 6, 2),(2, \leq 8),( \leq 5,3),( \leq 4,4)$ \cite{ES66,EF96,L97,NS21,S04,S08,Spre,SO20,W12}. 
There are analogous upper bounds on $s$-distance set in various spaces like on the Euclidean sphere \cite{DGS77} or association schemes \cite{BIb,D73}. 
Largest $s$-distance sets often have good combinatorial structures; in particular, tight $s$-distance sets are closely related to designs in the sphere or certain association schemes \cite{D73,DGS77}. 

There is a similar upper bound for a set in $\R^d$ with only $s$ squared distances distinct modulo some prime ideal $\p A_\p$ as we show \eqref{eq:bound} bellow.    
An extension field $K$ of rationals $\mathbb{Q}$ is an {\it algebraic number field} if the degree $[K:\mathbb{Q}]$ is finite.  
The {\it ring of integers $\mathcal{O}_K$} is the ring consisting of all algebraic integers in $K$.  
It is well known that any prime ideal of $\mathcal{O}_K$ is maximal. 
We always suppose $K$ is embedded into $\mathbb{R}$, and regard $K \subset \mathbb{R}$.  
For a commutative ring $A$, $(A,\mathfrak{m})$ is a {\it local ring} if $A$ has a unique maximal ideal $\mathfrak{m}$. 
It is well known that for a commutative ring $A$ and its maximal ideal $\mathfrak{p}\subset A$, we can construct a local ring $(A_{\mathfrak{p}}, \mathfrak{p} A_{\mathfrak{p}})$.  
For $A=\mathcal{O}_K$, the local ring is 
\[
A_\mathfrak{p}=S^{-1} A=\{a/s \in K \mid a \in A, s \in S\},
\]
where $S=A \setminus \mathfrak{p}$. 
Its unique maximal ideal is $\mathfrak{p} A_{\mathfrak{p}}$, which is the ideal of 
$A_{\mathfrak{p}}$ generated by the elements of $\mathfrak{p}$. It is noteworthy that $A_\mathfrak{p}$ is a principal ideal domain and the natural map $f: A/\mathfrak{p} \rightarrow A_\mathfrak{p}/ \mathfrak{p} A_\mathfrak{p}: x+\p \mapsto x+\p A_\p$ is a field isomorphism. 
If $D(X) \subset A_\mathfrak{p}$ for some integer ring $A=\mathcal{O}_K$ and a prime ideal $\mathfrak{p} \subset \mathcal{O}_K$, we define 
\[
D_{\mathfrak{p}}(X) = \{a \mod{\mathfrak{p}A_{\mathfrak{p}}} \mid a \in D(X)\}. 
\]
If $D(X) \subset A$ holds, one has 
\[
D_{\mathfrak{p}}(X) = \{a \mod{\mathfrak{p}} \mid a \in D(X)\}. 
\]
Nozaki \cite{NX} proved that if $0 \not\in D_{\p}(X)$ and $|D_{\p}(X)|=s$, then 
\begin{equation} \label{eq:bound}
|X| \leq \binom{d+s}{s}+\binom{d+s-1}{s-1}. 
\end{equation} 
We call $X$ an {\it $s$-distance set modulo $\mathfrak{p}A_\mathfrak{p}$} ({\it resp}.\ $\mathfrak{p}$)
if $D(X)\subset A_\p$ ({\it resp}.\ $\mathcal{O}_K$), $0 \not\in D_\p(X)$, and $|D_\p(X)|=s$. 
Notice that $D_\p(X)$ does not include $0$ if and only if $D(X)$ is a subset of $A_\p \setminus \p A_\p$. 
The term {\it modular $s$-distance set} is also used if we do not need to specify $A_\p$.  
 An $s$-distance set modulo $\p A_\p$ is said to be {\it tight} if the set attains the bound \eqref{eq:bound}. 

The number of the distances in $X$ is the same as that of $\sqrt{r}X=\{\sqrt{r} x \colon\, x \in X\}$ for any $r>0$. 
We consider under what conditions \( D(\sqrt{r}X) \) can become a subset of \( A_\p \setminus \p A_\p \) when \( D(X) \) is a subset of \( K \).
For $D(\sqrt{r} X)$ to remain a subset of $K$, $r$ must be an element of $K$.
For $0\ne a \in K$, let ${\rm ord}_\mathfrak{p}(a)$ denote the integer $n$ such that $a A_\mathfrak{p}=(\mathfrak{p} A_\mathfrak{p})^n$, where $a A_\mathfrak{p}$ and $(\mathfrak{p} A_\mathfrak{p})^n$ are fractional ideals of $A_\mathfrak{p}$ (finitely generated  $A_\mathfrak{p}$-submodule of $K$) and ${\rm ord}_\mathfrak{p}(0)=\infty$.    
Our assumption $D(\sqrt{r}X)\subset A_\p \setminus \p A_\p$ implies that each $a \in D(\sqrt{r}X)$ must have order $0$. 
Therefore, 
${\rm ord}_\p(a)$ does not depend on $a \in D(X)$ if and only if  
there exists $r \in K$ such that our assumption $0 \not\in D_\p(\sqrt{r}X)\subset A_\p \setminus \p A_\p$ is satisfied. 
Here, $r\in K$ should satisfy ${\rm ord}_\p(r)=-{\rm ord}_\p(a)$ for each $a \in D(X)$. 
Two $s$-distance sets $X$ and $Y$ modulo $\mathfrak{p} A_{\mathfrak{p}}$ are said to be \textit{similar} if there exist $r \in K$ and a congruence transformation $\sigma$ such that  $\operatorname{ord}_{\mathfrak{p}}(r) = 0$ and $Y = \sigma(\sqrt{r} X)$.
We deal with modular $s$-distance sets up to this specific similarity.  
We may suppose $1 \in D(X)$ after this rescaling.

It is noteworthy that if $X$ is an $s$-distance set modulo $\p A_\p$, then 
there exists a unit $u \in A_\p\setminus \p A_\p$ such that $\sqrt{u} X$ is an $s$-distance set modulo $\p$. 
An $s$-distance set modulo $\p$ is clearly an $s$-distance set modulo $\p A_\p$ because of $\mathcal{O}_K \subset A_\p$ and $\p \subset \p A_\p$. 
They can be understood as essentially the same concept up to similarity.

A natural problem concerning $s$-distance set modulo $\p A_\p$ is to determine when a tight set exists. In this paper, we focus on the case $s=1$, where the upper bound in \eqref{eq:bound} is $d+2$. 
The regular simplex is a $1$-distance set in $\R^d$, and it has $d+1$ points. The regular simplex with side length 1 is a $1$-distance set modulo $\p A_\p$ for any $K$ and $\p$. 
If there exists no tight 1-distance set modulo $\p A_\p$, then the regular simplex is one of the largest sets. 
Rosenfeld~\cite{R97} showed a related result that  there exists a tight $1$-distance set modulo $2 \Z$
if and only if $d+2 \equiv 0 \pmod{4}$ (see also \cite{Sp}). 
We generalize this result for any algebraic number field $K$. 
More precisely, when the characteristic of $A_\p/\p A_\p$ is 2, there exists a tight $1$-distance set modulo $\p A_\p$
if and only if $d+2 \equiv 0 \pmod{4}$. 
When the characteristic of $A_\p/\p A_\p$ is $p\ne 2$, there exists a tight $1$-distance set modulo $\p A_\p$
if and only if $d+2 \equiv 0 \pmod{p}$.  
This can be considered as one of the applications of finite fields, especially using characteristic, to discrete geometry.

This paper is organized as follows.
In Section \ref{sec:exist}, for any algebraic number field $K$ and prime ideal $\p$, 
we determine the dimensions where there exists a tight $1$-distance set modulo $\p A_\p$. 
In Section \ref{sec:example}, we discuss examples of tight $1$-distance sets modulo $\p A_\p$. 
Initially, we prove that for an $s$-distance set $X$ modulo $\p A_\p$, there exists 
a congruence transformation $\sigma$ such that the coordinates of each point in $\sigma(X)$ are contained in some finite extension field $L \supset K$. 
From this fact, $X$ can be modified to a non-similar modular $s$-distance set $X'$ in $ L^d\subset \R^d$. This modification yields infinitely many non-similar tight modular $1$-distance sets.  
 It is noteworthy that tight $s$-distance sets modulo $\p A_\p$ are known only for $s=1$. 
Tight 1-distance sets modulo $\p A_\p$ with only $2$ Euclidean distances are characterized from the viewpoint of representations of simple graphs. Tight examples obtained from Euclidean 2-distance sets that contain the regular simplex are presented.

\section{Dimensions where tight 1-distance sets modulo $\p A_\p$ exist } \label{sec:exist}
In this section, we determine the dimensions where there exist tight 1-distance sets modulo $\p A_\p$. 
We continue to use the same notation as before. 
Throughout this paper, let $J$ denote the all-ones matrix and $I$ the identity matrix.   
It is well known that for a prime ideal $\p \subset \Oc_K$, the set $\p \cap \Z $ is a prime ideal $p \Z$ of $\Z$ for some prime number $p$. For this case,  $\p \subset \Oc_K$ is called a  prime ideal lying above $p\Z$. 

The following theorem can be proved using the same method as presented in \cite{R97}. 
\begin{theorem} \label{thm:gen_d+2}
Let $p$ be a prime number in $\Z$. 
Let $A$ be the ring of integers of an algebraic number field. 
Let $\p\subset A$ be a prime ideal lying above $p\Z$. 
If there exists a $1$-distance set modulo $\p A_\p$ in $\R^d$ with $d+2$ points, then $d+2\equiv 0 \pmod{p}$.   
\end{theorem}
\begin{proof}
Let $X$ be a tight 1-distance set modulo $\p A_\p$ in $\R^d$. 
We may suppose $D_\p(X)=\{1\}$. 
Let $x_1,\ldots, x_n$ be the points in $X$, where $n=d+2=|X|$. For $u_i=x_i-x_1$ ($2\leq i \leq n$), we consider the Gram matrix $M=(2 \langle u_i, u_j \rangle)_{i,j=2,\ldots, n}$, where $\la ,\ra$ is the usual inner product of $\mathbb{R}^d$. 
The rank of $M$ is the dimension of the subspace spanned by $\{u_2, \ldots, u_n\}$, which is at most $d$. 
Thus, the rank of $M$ is less than $n-1$, and hence the determinant of $M$ is 0. 
From our assumption, $||u_i||^2=||x_i-x_1||^2 \equiv 1 \pmod{\p A_\p}$ and $||u_i-u_j||^2=||x_i-x_j||^2\equiv 1 \pmod{\p A_\p}$, and hence 
\[
2\la u_i, u_j \ra =||u_i||^2+||u_j||^2-||u_i-u_j||^2\equiv \begin{cases}
2 \pmod{\p A_\p} \text{ if $i=j$},\\
1 \pmod{\p A_\p} \text{ if $i\ne j$}.
\end{cases}
\]
Therefore, 
\[
0=\det (M) \equiv \det (J+I)=n =d+2 \pmod{\p A_\p}, 
\]
and $d+2 \equiv 0 \pmod{p}$ because of $\p A_\p\cap \Z=p\Z$. 
\end{proof}
Theorem~\ref{thm:gen_d+2} shows the necessary condition for the existence of tight 1-distance sets modulo $\p A_\p$. For an odd prime $p$, the condition is also sufficient as follows. 
\begin{theorem} \label{thm:main_odd}
Let $p$ be an odd prime number. 
Let $A$ be the ring of integers of an algebraic number field. 
Let $\p$ be a prime ideal of $A$ lying above $p \Z$. 
Then, the following are equivalent. 
\begin{enumerate}
    \item There exists a $1$-distance set $X \subset \mathbb{R}^d$ modulo $\p A_p$ with $d+2$ points.
    \item $d+2\equiv 0 \pmod{p}$. 
\end{enumerate}
\end{theorem}
\begin{proof}
$(1) \Rightarrow (2)$ follows from Theorem~\ref{thm:gen_d+2}. 
For $(2) \Rightarrow (1)$, the following set $X$ is an example of a tight 1-distance set modulo $\p A_\p$: 
\begin{equation} \label{ex:regular}
X=\{\frac{1}{\sqrt{2}} e_1, \ldots, \frac{1}{\sqrt{2}} e_d, \frac{1}{\sqrt{2}} x, \frac{1}{\sqrt{2}}y \},
\end{equation}
where $e_i$ is the vector with the $i$-th entry 1 and the other entries 0, $x=(\alpha, \ldots ,\alpha)$, $y=(\beta,\ldots, \beta)$, 
$\alpha=(1+\sqrt{d+1})/d$, and $\beta=(1-\sqrt{d+1})/d$. 
For this set $X$, the set of squared-distances is $D(X)=\{1,1+(d+2)/d\}$.  The greatest common divisor of $d+2$ and $d$ is at most 2, and $d$ is coprime to $p$. Thus $D(X) \subset A_\p$ and $D_\p(X)=\{1\}$. 
\end{proof}
\begin{remark}
The notation is defined in Theorem \ref{thm:main_odd}. 
If $d+2 \equiv 0 \pmod{p}$, then the regular simplex in $\R^d$ with side length 1 and its center also form a tight $1$-distance set modulo $\p A_\p$. Indeed, 
the other squared distance is $1-(d+2)/(2(d+1))$ and two integers $d+2$ and $2(d+1)$ are coprime. Thus, $D_\p(X)=\{1\}$ holds.  
\end{remark}
The following is the main theorem, which is a generalization of Rosenfeld's result \cite{R97}. 
\begin{theorem} \label{thm:order2}
Let $A$ be the ring of integers of an algebraic number field, and 
let $\p$ be a prime ideal of $A$ lying above $2\Z$.  
Then, the following are equivalent. 
\begin{enumerate}
    \item There exists a $1$-distance set $X \subset \mathbb{R}^d$ modulo $\p A_\p$ with $d+2$ points.
    \item $d\equiv 2 \pmod{4}$. 
\end{enumerate}
\end{theorem}
The following lemma is stated to prove Theorem \ref{thm:order2}. 
\begin{lemma} \label{lem:to_prove}
Let $A$ be the ring of integers of an algebraic number field, and let $\p$ be the prime ideal of $A$ lying above $2\Z$. Let $n$ be an odd positive integer, and suppose that a symmetric matrix $M = (m_{ij}) \in \mathrm{M}_n(A_{\p})$ satisfies
\begin{align}
m_{ij} \equiv
\begin{cases}
2 \pmod{2\p A_{\p}} & \text{if } i = j,\\
1 \pmod{\p A_{\p}} & \text{if } i \ne j.
\end{cases}
\label{eq:matrix_mij}
\end{align}
Then $\det M \equiv n+1 \pmod{2\p A_{\p}}$.
\end{lemma}

\begin{proof}
Since $\det(J + I) = n + 1$, it suffices to show that
\begin{align}
\det M \equiv \det(J + I) \pmod{2\p A_{\p}}.
\label{eq:det_proof}
\end{align}
In general, for a square matrix $N = (n_{ij}) \in \mathrm{M}_n(A_{\p})$ and a permutation $\sigma \in S_n$, define
\[
N(\sigma) = \prod_{k=1}^n n_{k,\sigma(k)}.
\]
If $N$ is symmetric, then 
\[
N(\sigma)=\prod_{k=1}^n n_{k,\sigma(k)}
=\prod_{k=1}^n n_{\sigma^{-1}(k),k}=N(\sigma^{-1}).
\]
Define $M' = (m'_{ij}) \in \mathrm{M}_n(A_{\p})$ by
\[
m'_{ij} =
\begin{cases}
\frac{1}{2} m_{ii} & \text{if } i = j,\\
m_{ij} & \text{if } i \ne j. 
\end{cases}
\]
Then
\[
m'_{ij} \equiv 1 \pmod{\p A_{\p}} \quad (1 \leq i, j \leq n).
\]
Let $f_{\sigma}$ be the number of fixed points of $\sigma \in S_n$. 
For $\sigma \in S_n$ with $\sigma = \sigma^{-1}$, $f_{\sigma}$ is nonzero since $n$ is odd. Since $M'(\sigma) \equiv 1 \pmod{\p A_{\p}}$ (equivalently $2M'(\sigma) \equiv 2 \pmod{2\p A_{\p}}$) by the above congruence, we have
\begin{equation}
M(\sigma) = 2^{f_{\sigma}} M'(\sigma)=2^{f_{\sigma}-1}\cdot 2M'(\sigma)  
\equiv 2^{f_{\sigma}} = (J + I)(\sigma)\pmod{2\p A_{\p}}. \label{eq:M_sigma}    
\end{equation}
For $\sigma \in S_n$ with $\sigma \neq \sigma^{-1}$, since $M(\sigma) \equiv (J + I)(\sigma) \pmod{\p A_{\p}}$, we have
\begin{equation}
M(\sigma) + M(\sigma^{-1}) = 2M(\sigma)
\equiv 2(J + I)(\sigma) 
= (J + I)(\sigma) + (J + I)(\sigma^{-1}) \pmod{2\p A_{\p}}. \label{eq:M_sigma_symm}
\end{equation}
Using the definition of the determinant together with \eqref{eq:M_sigma} and \eqref{eq:M_sigma_symm}, we obtain the desired result \eqref{eq:det_proof}.
\end{proof}

\begin{proof}[Proof of Theorem \ref{thm:order2}]
$(2) \Rightarrow (1)$ is already proved in \cite{R97}, namely for $d\equiv 2 \pmod{4}$, the regular simplex with side length 1 and its center is a tight 1-distance set modulo $2\Z_{(2)}\subset \p A_\p$. 
The example in \eqref{ex:regular} is also 
a tight 1-distance set modulo $2 \Z_{(2)}\subset \p A_\p$ for $d\equiv 2 \pmod{4}$. 

We prove $(1) \Rightarrow (2)$. 
Let $X$ be a tight 1-distance set modulo $\p A_\p$ in $\R^d$. 
We may suppose $D_\p(X)=\{1\}$. Let $M=(2\la u_i, u_j \ra)_{i,j=2,\ldots n}=(m_{ij}) \in \mathrm{M}_{n-1}(A_\p)$ be defined as in the proof of Theorem~\ref{thm:gen_d+2}, where $n=d+2$. 
From Theorem \ref{thm:gen_d+2}, $n-1$ is odd. 
For $i\ne j$, it is already proved that $m_{ij}\equiv 1 \pmod{\p A_\p}$.
For $i= j$, since $||u_i||^2 \equiv 1 \pmod{\p A_\p}$ holds, we have  $m_{ii}= 2||u_i||^2 \equiv 2 \pmod{2\p A_\p}$. Therefore, by Lemma \ref{lem:to_prove}, 
 one has  
\[
0=\det(M) \equiv  n =d+2 \pmod{2\p A_\p},
\]
and $d\equiv 2 \pmod{4}$ since $2 \p A_\p \cap \mathbb{Z}=4\mathbb{Z}$. 
\end{proof}


In the proofs of Theorems \ref{thm:gen_d+2} and \ref{thm:order2}, the Gram matrix $M=(2\langle u_i,u_j \rangle)$ carries a factor of $2$ in each entry, representing twice of the inner product $\langle u_i,u_j\rangle$. 
The key idea of the proofs is to choose the modulus as the largest ideal that does not contain 2.  
This necessitates differentiating the proofs based on whether the characteristic of the finite field $A_\p/\p A_\p$ is an odd prime or not.

 Graham, Rothschild, and Straus \cite{GRS74} proved that 
there exists a tight 1-distance set modulo $2 \mathbb{Z}$ in $\mathbb{R}^d$ such that $D(X)$ is a set of squared odd integers if and only if $d+2\equiv 0 \pmod{16}$. 
Notice that if $D(X)$ is a set of squared odd integers, then
$a\equiv 1,9 \pmod{16}$ for each $a \in D(X)$.
This result is generalized as the following theorem. 
\begin{theorem}
    Let $A$ be the ring of integers of an algebraic number field, and let $\p$ be a prime ideal of $A$ lying above $2\Z$. 
Then, the following are equivalent. 
\begin{enumerate}
    \item There exists a $1$-distance set $X \subset \mathbb{R}^d$ modulo $\p A_\p$ with $d+2$ points such that $a\equiv 1, 9 \pmod{8 \p A_\p}$ for each $a \in D(X)$.
    \item $d+2\equiv 0 \pmod{16}$. 
\end{enumerate}
\end{theorem}
\begin{proof}
     Graham, Rothschild, and Straus \cite{GRS74} provided an example satisfying the condition of (1) modulo $2\Z_{(2)} \subset \p A_\p $ for each $d+2 \equiv 0 \pmod{16}$.  

   We prove $(1) \Rightarrow (2)$. The notation we use is the same as in the proof of Lemma \ref{lem:to_prove}. 
   Let $M=(2\la u_i, u_j \ra)_{i,j=2,\ldots n} \in \mathrm{M}_{n-1}(A_\p)$ be defined as in the proof of Theorem~\ref{thm:gen_d+2}, where $n=d+2$ and $n-1$ is odd. Note that $8 \p A_\p \cap \Z=16 \Z$. The matrix entries satisfy
   \[
2\la u_i, u_j \ra =||u_i||^2+||u_j||^2-||u_i-u_j||^2 \equiv
\begin{cases}
2 \pmod{8 \p A_{\p}} & \text{if } i = j,\\
1,9 \pmod{8 \p A_{\p}} & \text{if } i \ne j,
\end{cases}
\]
since $||u_i||^2\equiv ||u_j||^2\equiv 1,9 \pmod{8 \p A_\p}$ and 
\[
||u_i-u_j||^2\equiv \begin{cases}
    0 \pmod{8 \p A_\p}& \text{if } i = j,\\
    1,9 \pmod{8 \p A_\p}& \text{if } i \ne j.
\end{cases}
\]
Let $f_{\sigma}$ be the number of fixed points of $\sigma \in S_n$. 
From $2 = 2 \cdot 1 \equiv 2 \cdot 9 \pmod{8 \mathfrak{p} A_\mathfrak{p}}$, it follows that any integer of the form $2^s 1^t 9^k$ is congruent to $2^s$ modulo $8 \mathfrak{p} A_\mathfrak{p}$ when $s \geq 1$.  
Therefore, one has $2M(\sigma)\equiv 2\cdot 2^{f_{\sigma}} \pmod{8 \p A_\p}$. 
Moreover, if $f_\sigma$ is nonzero, then $M(\sigma)=2^{f_{\sigma}}$.

For $\sigma \in S_{n-1}$ with $\sigma = \sigma^{-1}$, 
\begin{align*}
    M(\sigma)\equiv 2^{f_{\sigma}} \equiv (J+I)(\sigma) \pmod{8 \p A_\p}
\end{align*}
since $f_{\sigma}$ is nonzero. 
For $\sigma \in S_{n-1}$ with $\sigma \ne \sigma^{-1}$, 
\begin{align*}
    M(\sigma)+M(\sigma^{-1})=2 M(\sigma) \equiv 2 \cdot 2^{f_{\sigma}} = 2(J+I)(\sigma)= (J+I)(\sigma)+(J+I)(\sigma^{-1})\pmod{8 \p A_\p}.
\end{align*}
Therefore, one has
\[
0=\det(M) \equiv \det(J+I) = n =d+2 \pmod{8\p A_\p},
\]
and $d\equiv 2 \pmod{16}$. 
\end{proof}

\section{Tight 1-distance set modulo $\p A_\p$} \label{sec:example}
In this section, we provide examples of tight 1-distance sets modulo $\p A_\p$. 
We use the same notation $K$, $\Oc_K=A$, $\p$, and $A_\p$ as in the previous sections. 
Initially, we will verify whether the number of modular distances decreases upon field extension $L \supset K$. 
Let $P\subset B=\mathcal{O}_L$ be a prime ideal lying above a prime ideal $\mathfrak{p} \subset A=\mathcal{O}_K$, namely  $\mathfrak{p}=P \cap \mathcal{O}_K$ holds.  
If $D(X) \subset A_\mathfrak{p}$ holds, then  $D(X) \subset B_P$ and $|D_\mathfrak{p}(X)|=|D_P(X)|$ hold. 
Indeed, the natural homomorphism $\alpha + \mathfrak{p}A_\mathfrak{p} \mapsto 
 \alpha + PB_P$ is injective. 
 From this discussion, if every element of $D(X)=\{a_1,\ldots, a_s\}$ is an algebraic number, then it is sufficient to select the field $K=\Q(a_1,\ldots, a_s)$ for the reduction of the number of modular distances. 
 
An $s$-distance set modulo $\p A_\p$ can be modified to a non-similar $s$-distance set modulo $PB_P$ for some ring $B\supset A$ and a prime ideal $P\subset B$ with $\p=P\cap A$. 
From the following lemma, for an $s$-distance set $X$ modulo $\p A_p$, 
we may suppose that the coordinates of each point in $X$ are contained in some finite extension field $L \supset K$. 
\begin{lemma} \label{lem:L}
Let $X$ be a finite set in $\R^d$ whose squared distances $D(X)$ are contained in an algebraic number field $K$.
Then, there exist a congruence transformation $\sigma$ and a finite extension field $L$ of $K$ such that $L \subset \R$ and $L$ contains all coordinates of each point of $\sigma(X)$. 
\end{lemma}
\begin{proof}
Let $M$ be the squared-distance matrix $M=(||x-y||^2)_{x,y\in X}$ of $X$. 
From our assumption, each element of $M$ belongs to $K$. Let $P$ be the matrix $P=I-(1/|X|) J$. 
By direct calculation, $N=(-1/2)PMP$ is the Gram matrix of $\sigma(X)$ with some congruence transformation $\sigma$ \cite{N81}. 
Indeed, the $(x,y)$-entry of $N$ can be written as 
$\langle x-w,y-w\rangle$, where $w=(1/|X|)\sum_{z \in X}z$.  
Note that each entry in the Gram matrix $N$ belongs to $K$.  
By the Cholesky factorization algorithm \cite[Theorem 10.9]{Hb}, there exists an $n\times d$ matrix $F$ such that $N =FF^\top$. 
The computation of each entry of $F$, as defined in the algorithm, establishes that each entry belongs to a certain finite extension field $L$ of $K$. The row vectors in $F$ can be regarded as the coordinates of $\sigma(X)$. 
This implies the lemma. 
\end{proof}
From the following theorem, 
infinitely many non-similar modular $s$-distance sets can be obtained from a given $s$-distance set modulo $\p A_\p$ while maintaining the cardinality.  
\begin{theorem} \label{thm:make_ex}
Let $X=\{x_1,\ldots, x_n\} \subset \mathbb{R}^d$ be an $s$-distance set modulo $\p A_\p$, where $A=\Oc_K$ and $\p$ is a prime ideal of $A$. 
From Lemma \ref{lem:L}, we may suppose that  there exists a finite extension field $L\supset K$ such that   $L$ contains all coordinates $x_{ij}$ of each point $x_i=(x_{i1},\ldots, x_{id})\in X$. 
Let $P\subset B=\Oc_L$ be a prime ideal lying above $\p\subset \Oc_K$. 
Let $t=\max\{0,-\min_{i,j}{\rm ord}_{P}(2x_{ij})\}$. Then, 
for any $a_i=(a_{i1},\ldots, a_{id})$ with $a_{ij} \in (PB_P)^{t+1}$, the set $X'=\{x_1+a_1,\ldots,x_n+a_n\}$ is 
an $s$-distance set modulo $P B_P$. 
\end{theorem}
\begin{proof}
For any $x_i+a_i,x_j+a_j \in X'$, the squared distance between them is 
\begin{align*}
||(x_i+a_i)-(x_j+a_j)||^2&=||(x_i-x_j)+(a_i-a_j)||^2\\
&=||x_i-x_j||^2+2\langle x_i-x_j,a_i-a_j \rangle+||a_i-a_j||^2\\
&\equiv ||x_i-x_j||^2 \pmod{PB_P}
\end{align*}
from our assumption because of ${\rm ord}_{P}(2 (x_{ik}-x_{jk}))\geq \min\{ {\rm ord}_{P}(2x_{ik}),{\rm ord}_P (2x_{jk}))\} \geq -t$ and ${\rm ord}_{P}(a_{ik}-a_{jk})\geq t+1$. Thus, the assertion follows. 
\end{proof}
\begin{remark}
From the tight 1-distance set modulo $p \mathbb{Z}_{(p)}$ defined in \eqref{ex:regular},  
 we can obtain infinitely many non-similar tight modular 1-distance sets.
\end{remark}

In the remaining part of this section, we provide tight 1-distance sets modulo $\p A_\p$ having only two Euclidean distances. 
For $X\subset \R^d$ with $D(X)=\{ a_1,\ldots, a_s \}$, the following value $t_i$ is called the {\it LRS ratio}: 
\[
t_i=\prod_{j\ne i} \frac{a_j}{a_j-a_i}  
\]
for $i\in \{1,\ldots, s\}$ \cite{LRS77,N11}. Note that $\sum_{i=1}^s t_i=1$ \cite{MN11}. 
If $|X|> \binom{d+s-1}{s-1}+\binom{d+s-2}{s-2}$ holds, 
then the values $t_i$ are algebraic integers \cite[Corollary 4.2]{NX}.   
If $t_i$ is an algebraic integer for each $i$, then there does not exist $K$ such that  
$|D_\p(X)|<s$ for some $\p \subset \Oc_K$ \cite[Equation (4.1)]{NX}. The converse is not true in general, but it is true for $s=2$.

\begin{theorem} \label{thm:2-dis}
Let $X\subset \R^d$ be a 2-distance set with $D(X)=\{1,a \}$. Suppose the squared distance $a$ is an algebraic number. Let $K=\Q(a)$ and $\Oc_K$ the ring of integers of $K$. 
The LRS ratio $t=1/(1-a)$ is not an algebraic integer if and only if there exists a prime ideal $\p \subset \Oc_K$ such that $|D_\p(X)|=1$. 
\end{theorem}
\begin{proof}
We only prove that $t=1/(1-a) \not \in \Oc_K$ implies that there exists $\p \subset \Oc_K$ such that $|D_\p(X)|=1$. 
Let $(1-a)\Oc_K$ be the fractional ideal generated by $1-a$. We can obtain the prime factorization  $(1-a)\Oc_K=\p_1^{c_1}\cdots \p_r^{c_r}$. 
Since $t=1/(1-a) \not \in \Oc_K$, there exists $i$ such that $c_i \geq 1$. Put $\p=\p_i$ for such $i$, and then
$(1-a)A_\p=\p^{c_i} A_\p\subset \p A_\p$. 
This implies 
$a\equiv 1 \pmod{\p A_\p}$, and $D_\p(X)=\{1\}$.
\end{proof}

A 2-distance set is identified with a representation of a simple graph. If a 2-distance set $X$ in $\R^d$ has size at least $d+2$, then $X$ is a minimal-dimensional representation of some simple graph \cite{ES66,R10}. 

\begin{theorem}\label{thm:a}
Let $X \subset \mathbb{R}^d$ be a 2-distance set with $D(X)=\{1,a \}$. If $|X| \geq d+2$, then $a$ is an algebraic number.  
\end{theorem}
\begin{proof}
Let $A$ be the adjacency matrix of the simple graph corresponding to the 2-distance set $X$, and  
the squared-distance matrix of $X$ is equal to $M=aA+ A'$, where $A'=J-I-A$. 
Let $P=I-(1/|X|)J$, which is the orthogonal projection matrix onto the space $j^{\perp}$, where $j^{\perp}$ is the space perpendicular to the all-ones vector $j$. 
Then, $(-1/2)PMP$ is the Gram matrix of $\sigma(X)$ with some congruence transformation $\sigma$ and the rank of $-PMP$ is at most $d$ \cite{N81}. 
Given that $|X| \geq d+2$, the smallest eigenvalue of $-PMP = -(aPAP + PA'P)$ on $j^\perp$ is 0. 
Since $PAP$ and $PA'P=-P-PAP$ have the same eigenspaces (these matrices commute), the distance $a$ is expressed as $a = (\lambda+1)/\lambda$, where $\lambda$ is the smallest eigenvalue of $PAP$ on $j^\perp$.  
All entries of $PAP$ are rational, and hence $\lambda$ and $a$ are algebraic numbers. 
\end{proof}
From Theorems \ref{thm:2-dis} and \ref{thm:a}, 
a tight modular 1-distance set with two Euclidean distances is characterized by the property that its LRS ratio is not an algebraic integer. 
If $X$ is on a sphere and $|D_\p(X)|=1$ for some prime ideal $\p$, then $|X| \leq d+1$, which is proved by the same manner as in \cite{NX}. 
A non-spherical 2-distance set with $d+2$ points is characterized as Type (5), see the definition of Type (5) in \cite[Theorem 2.4]{NS12}.  
There are many simple graphs of Type (5) whose LRS ratio is not an algebraic integer. 
It is easy to obtain such graphs from random graphs.

A tight 1-distance set modulo $2\mathbb{Z}$ given in \cite{GRS74,R97} is a regular simplex and its center, which is a 2-distance set. 
Nozaki and Shinohara \cite{NS20} investigated a 2-distance set that contains a regular simplex. 
Let $\mathcal{R}$ be the $d$-dimensional regular simplex. 
The set $\mathcal{R}$ can be expressed by $\mathcal{R}=\{e_1,\ldots, e_{d+1}\} \subset H_d=\{x \in \R^{d+1} \mid \la j, x \ra =1\} \cong \mathbb{R}^d$, where $j$ is the all-ones vector and $e_i$ is the vector with the $i$-th entry 1 and the other entries 0.  
It is proved in \cite{NS20} that $\mathcal{R}\cup \{x\}$ is a 2-distance set in $H_d$ if and only if $x \in T_d(k,\beta)=\{(x_1,\ldots,x_{d+1}) \in H_d \mid \forall i, x_i \in \{c, c+\beta\},|N(x,c)|=k \}$, where $N(x,c)=\{i \mid x_i=c\}$, $k \in \{1,\ldots, d+1\}$,  
\[
c=\frac{1}{d+1}-\frac{d+1-k}{d+1} \beta, \text{ and } 
\beta=\begin{cases}
\frac{k\pm \sqrt{k(d+1)(d+2-k)}}{k(d+1-k)} \text{ if $2\leq k \leq d$},\\
1+\frac{2}{d} \text{ if $k=1$}, \\ 
-\frac{d+2}{2(d+1)} \text{ if $k=d+1$}. 
\end{cases}
\]
We determine when a 2-distance set containing the regular simplex $\mathcal{R}$ is a tight $1$-distance set modulo $\p A_\p$. 
\begin{theorem} \label{thm:4.7}
Let $\mathcal{R}=\{e_1,\ldots, e_{d+1}\}$ be the $d$-dimensional regular simplex in $H_d$, and $x \in T_d(k,\beta)\subset H_d$.  
Let $X=(1/\sqrt{2})(\mathcal{R} \cup \{x\})$, which is a 2-distance set. 
Let $K$ be an algebraic number field that contains the distances $D(X)$. 
If $K\ne \Q$,  
then the following are equivalent.  
\begin{enumerate}
    \item There exists a prime ideal $\p \subset A=\Oc_K$ such that $X$ is a tight $1$-distance set modulo $\p A_\p$. 
    \item $k\ne (d+2)/2$ or $d \not \equiv 0 \pmod{4}$. 
\end{enumerate}
\end{theorem}
\begin{proof}
By Theorems~\ref{thm:2-dis} and \ref{thm:a}, it suffices to determine when the LRS ratio is an algebraic integer.  Note that $D(X)=\{1,\beta+1\}$. Then, the LRS ratio is $t=-1/\beta$. 
From our assumption $K\ne \Q$, we have $k\ne 1,d+1$. 
For $2\leq k \leq d$, we can express
\[
t=-\frac{1}{\beta}=\frac{k}{d+2}\pm \frac{\sqrt{k(d+1)(d+2-k)}}{d+2}. 
\]
It is well known that for $L=\mathbb{Q}(\sqrt{r})$ with a square-free integer $r$, one has 
\[
\mathcal{O}_L=\begin{cases}
\mathbb{Z}+ \frac{1+\sqrt{r}}{2} \mathbb{Z} \text{\quad  if $r \equiv 1 \pmod{4}$},\\ 
\mathbb{Z}+ \sqrt{r} \mathbb{Z} \text{\quad if $ r \equiv 2,3 \pmod{4}$}.
\end{cases}
\]
If $\sqrt{k(d+1)(d+2-k)}$ is not an integer and $t \in \Oc_K$, 
then $k/(d+2) \in (1/2) \mathbb{Z}$. 
From $2\leq k \leq d$, one has $k=(d+2)/2$ and $d$ is even. Moreover, we have
\[
\frac{\sqrt{k(d+1)(d+2-k)}}{d+2}=\frac{\sqrt{d+1}}{2}
\]
with $k=(d+2)/2$. Since $d+1$ is an odd integer, $d+1$ should be congruent to 1 modulo 4 to satisfy $t\in \Oc_K$ even if $d+1$ is not square-free. 
Therefore, $t\in \Oc_K$ if and only if $k=(d+2)/2$ and $d\equiv 0 \pmod{4}$. 
From Theorem \ref{thm:2-dis}, this theorem follows. 
\end{proof}
\begin{remark}
For $k=1$, it follows that $K= \Q$ and the LRS ratio $-1/\beta=-d/(d+2)$ is not an integer. Thus, for $k=1$, there exists a prime ideal $\p \subset \Oc_K$ such that $|D_\p(X)|=1$ by Theorem \ref{thm:2-dis}. This is the example \eqref{ex:regular}.  
For $k=d+1$, it follows that $K= \Q$ and the LRS ratio $-1/\beta=2(d+1)/(d+2)$ is not an integer. Thus, for $k=d+1$, there exists a prime ideal $\p \subset \Oc_K$ such that $|D_\p(X)|=1$ by Theorem \ref{thm:2-dis}. 
This corresponds to the regular simplex and its center.  
\end{remark}

\begin{remark} \label{rem:modif}
A tight 1-distance set $X$ modulo $\p A_\p$ presented in Theorem \ref{thm:4.7} can be obtained in the manner described in Theorem \ref{thm:make_ex} from the regular simplex and its center. 
Indeed, for the coordinates 
\[
X=\frac{1}{\sqrt{2}}\{ e_1,\ldots, e_{d+1}, (c,\ldots,c,c+\beta,\ldots, c+\beta)\},
\]
one has $D(X)=\{1,\beta+1\}$ and hence $\beta\equiv 0 \pmod{\p A_\p}$. From Theorems \ref{thm:main_odd} and \ref{thm:order2},
$\p$ contains $d+2$ but does not contain $d+1$. Therefore,
\[
c+\beta \equiv c=\frac{1}{d+1}-\frac{d+1-k}{d+1} \beta \equiv \frac{1}{d+1} \pmod{\p A_\p}, 
\]
and $(1/(d+1),\ldots, 1/(d+1))$ corresponds the center of the regular simplex $\{e_1,\ldots, e_{d+1}\}$. 
\end{remark}

Every $(d+2)$-point non-spherical 2-distance set with a non-algebraic integer LRS ratio is a tight modular 1-distance set.  
Many tight modular 1-distance sets can easily be derived from the 2-distance sets through Theorem \ref{thm:make_ex}.  
Our goal is to classify tight modular 1-distance sets while exploring possible frameworks for this classification.

\bigskip

\noindent
\textbf{Acknowledgments.} 
The author expresses gratitude to Akihiro Munemasa for suggesting the use of the Cholesky factorization algorithm to adjust the coordinates of Euclidean sets and for providing valuable comments on the initial draft of the paper. 
The author also thanks the two anonymous referees for their insightful comments on ambiguous aspects of the discussion. 
Furthermore, the author is particularly grateful to one of the referees for their suggestion, which greatly simplified the proof of Theorem \ref{thm:order2}. 
 The author was partially supported by JSPS KAKENHI Grant Numbers 18K03396, 19K03445, 20K03527, 22K03402, and 24K06688.

\end{document}